\documentclass[12pt]{amsart}
 \usepackage{oldgerm}

 \swapnumbers

 \newtheorem{theorem}{Theorem}
 \newtheorem{lemma}[theorem]{Lemma}
 \newtheorem{proposition}[theorem]{Proposition}
 \newtheorem{corollary}[theorem]{Corollary}
 \theoremstyle{definition}
 \newtheorem{remark}[theorem]{\bf Remark}
 \newtheorem{definition}[theorem]{Definition}

 \newcommand{\Z}{\mathbb Z}
 \newcommand{\R}{\mathbb R}
 \newcommand{\mS}{\mathbb S}

 \newcommand{\cU}{\mathcal U}
 \newcommand{\cL}{\mathcal L}
 \newcommand{\cM}{\mathcal M}

 \newcommand{\ga}{\gamma}
 \newcommand{\Ga}{\Gamma}
 \newcommand{\de}{\delta}

 \newcommand{\lV}{\left\Vert}
 \newcommand{\rV}{\right\Vert}

 \def \a{\alpha}
 
 \def \ga{\gamma}
 \def \Ga{\Gamma}
 \def \de{\delta}
 \def \e{\varepsilon}
 \def \la{\lambda}
 
 \def \vr{\varphi}

 \def \re{{\mathbb R}}
 
 \def \Z {{\mathbb Z}}
 
 \def \lv{\left\vert}
 \def \rv{\right\vert}
 \def \lV{\left\Vert}
 \def \rV{\right\Vert}

 \def\cS{{\mathcal S}}
 \def\cA{{\mathcal A}}
 \def\fA{{\textgoth A}}
 
 \def\AA{{\mathbb A}}
 \def\cU{{\mathcal U}}
 \def\fu{{\mathfrak u}}
 \def\cV{{\mathcal V}}
 \def\fv{{\mathfrak v}}

\begin{document}
 \title[Weak solutions]{Weak solutions of the Hamilton-Jacobi equation
 for Time Periodic Lagrangians} 
 \author[G. Contreras]{Gonzalo Contreras}
 \address{ CIMAT, A.P. 402, 3600, Guanajuato. Gto, M\'exico}
 \email{gonzalo@cimat.mx}
 \author[R. Iturriaga]{Renato Iturriaga}
 \address{ CIMAT, A.P. 402, 3600, Guanajuato. Gto, M\'exico}
 \email{renato@cimat.mx}
 \author[H. S\'anchez-Morgado]{H\'ector S\'anchez-Morgado}
 \address{ Instituto de Matem\'aticas, UNAM. Ciudad Universitaria C. P. 
          04510, Cd. de M\'exico, M\'exico.}
 \email{hector@math.unam.mx}
 \thanks{All authors were partially
 supported by  CONACYT-M\'exico grant $\#$ 28489-E}
 
 \thanks{{\bf Keywords:} Hamilton-Jacobi equation, weak KAM theory,
 action potential, periodic Lagrangians, Aubry-Mather theory,
 Ma\~n\'e's critical level, lagrangian graphs.}
 \thanks{{\bf AMS 2000 Math. Subject Classification:} 70H20, 37J50,
 70H08, 49L99, 37J05.}
 
 \begin{abstract}
 In this work we prove the existence of Fathi's weak KAM solutions for
 periodic Lagrangians and give a construction of all of them.
 \end{abstract}
 
 \maketitle
 \section{Introduction and statement of results}
 
 Let $M$ be a closed connected manifold, $TM$ its tangent bundle. Let
 $L:TM\times \R\to \R$ be a $C^{\infty}$ Lagrangian. We will assume for
 the Lagrangian the hypothesis of Mather's seminal paper~\cite{M}. The
 Lagrangian $L$ should be:
 
 \begin{enumerate}
 \item {\it Convex}. The Lagrangian $L$ restricted to $T_xM$, in linear 
 coordinates should have positive definite Hessian.
 \item {\it Superlinear}. For some Riemannian metric we have 
 $$
 \lim_{|v|\to \infty}\frac{L(x,v,t)}{|v|}=\infty,
 $$
 uniformly on $x$ and $t$.
 \item {\it Periodic}. The Lagrangian should be periodic in time, i.e.
 $$L(x,v,t+1)=L(x,v,t),$$
 for all $x,v,t$.
 \item {\it Complete}. The Euler Lagrange flow associated to the
   Lagrangian should be complete.
 \end{enumerate}

 Let ${\cM}(L)$ be the set of probabilities on the Borel
 $\sigma$-algebra of $TM\times\mS^1$
 that have compact support and are invariant under the Euler-Lagrange
 flow $\phi_{t}$.

 The {\it action} of $\mu\in {\cM}(L)$ is defined by
 \[A_{L}(\mu)=\int L\,d\mu.\]
 
 Mather defined the function $\alpha:H^1(M,\R)\to \R$ as
 \begin{equation}
 \alpha([\omega])=-\min\left\{\int (L-\omega)\,d\mu :\;
 \mu\in{\cM}(L)\right\}.    \label{coho}
 \end{equation}
 
 For any $k$ in $\R$ define the $(L+k)$-action of an absolutely 
 continuous curve 
 $\ga:[a,b]\rightarrow M$ as
 $$A_{L+k}(\ga)=\int_{a}^{b}(L+k)(\ga(\tau),\dot\ga(\tau),\tau) d\tau $$
 
 For $t$ in $\re$ we denote by $[t]$ the corresponding point in ${\mS}^1$.
 For any pair of points $(x,[s]), (y,[t])$ on $M\times{\mS}^1$ and $n$ a
 non negative integer, define
 $\mathcal{C}((x,[s]),(y,[t]);n)$ as 
 the set of absolutely continuous curves
 $\ga:[a,b]\to M$ with $\ga(a)=x$ and $\ga(b)=y$ such that $[a]= [s]$ and
 $[b]= [t]$, and the integer part of $b-a$ is $n$.

 Let $\Phi^n_k$ be the real function defined on 
 $M\times{\mS}^1\times M\times{\mS}^1$ as 
 
 $$\Phi^n_k((x,[s]),(y,[t]))=
 \min_{\ga \in\mathcal{C}((x,[s]),(y,[t]);n)}
 \{A_{L+k}(\ga)\}.$$
 so that  $\Phi^n_k=\Phi^n_0+kn$.
 
 Then the {\it action functional} is defined by
 $$\Phi_k=\inf_n\Phi^n_k,$$
 and the {\it Extended Peierls barrier} by
 $$h_k=\liminf_{n\to \infty}\Phi^n_k.$$
 Thus $\Phi_k\le h_k$.

 \bigskip
 A curve $\ga:[a,b]\to M$ will be called {\it closed } if
 $\ga(a)=\ga(b)$ and $b-a$ is an integer.
 In analogy to the autonomous case~\cite{Ma}, \cite{CDI},
 there is a critical value $c(L)$ given by the following proposition:

 \begin{proposition}\quad
 
 \begin{enumerate}
 \item If $k<c(L)$, then $\Phi_{k}((x,[s]),(y,[t]))=-\infty$, 
 for all $(x,[s]), (y,[t])$ on $M\times{\mS}^1$
 \item 
 $$c(L)= \min \{k: \int_{\ga}L+k \ge 0 {\mbox{ for all closed curves
     $\ga $}}\}$$
 \item If $k\geq c(L)$, then $\Phi_{k}((x,[s]),(y,[t]))>-\infty$
   for all $(x,[s]), (y,[t])$ on $M\times{\mS}^1$.

\item In terms of Mather's $\alpha$ function we have 
 
 \begin{eqnarray}\label{erg.charac}
     c(L)& = &  
-\min \Bigl\{\int Ld\mu :\mu{\mbox { is an invariant probability}}\Bigr\}\\
     & =  &\alpha (0)
   \end{eqnarray}
 
   Invariant probabilities realizing the infimum above are called
   {\it minimizing measures}.
 \end{enumerate}               \label{basic}
 \end{proposition}   
   \bigskip
 
  From now on, set  $c=c(L)$.
 
  In contrast with the autonomous case, the action potential~$\Phi_c$
  may fail to be continuous and to satisfy the triangle inequality.
  However, for the extended Peierls barrier we shall prove the following:

 \begin{proposition}\label{barrera}\quad
     \begin{enumerate}
   \item If $k<c$, $h_k\equiv -\infty$.
   \item If $k>c$, $h_k\equiv \infty$.
   \item $h_c$ is finite.
   \item $h_c((x,[s]),(z,[\tau ]))\le
     h_c((x,[s]),(y,[t]))+\Phi_c((y,[t]),(z,[\tau ]))$.
   \item $h_c$ is Lipschitz.  
   \end{enumerate}
 \end{proposition}

 Let $H(x,p,t)$ be the Hamiltonian associated to the Lagrangian;
 $$H:T^*M\times\re\to \R$$
 \begin{equation}
   \label{eq:hamiltonian}
   H(x,p,t)=\max_{v\in T_xM}pv-L(x,v,t)
 \end{equation}
 
 In~\cite{gafa} the critical  value
 or $\alpha (0)$ for the autonomous case is characterized as 
 
 \begin{align*}
 c(L) &=\inf_{f\in C^{\infty}(M,\R)}\sup_{x\in M}H(x,d_{x}f)\\
      &=\inf\{k\in\R:\;\mbox{\rm there exists $f\in C^{\infty}(M,\R)$
 such that $H(df)<k$}\},
 \end{align*}

 This can be restated in physical terms, by saying that $c(L)$
 is either the infimum of the values of $k\in\R$ for which
 there is an exact Lagrangian graph with energy less than $k$, 
 or the infimum of the values
 of $k\in\R$ for which there exist smooth solutions of the
 Hamilton-Jacobi inequality $H(df)<k$.
 
 The second interpretation has a natural  generalization. We will prove
 in section~\ref{Hamilton Jacobi} the following 
 
 \begin{theorem}
 \label{gafaperiodico}
 The critical value, $c(L)$ or $\alpha (0)$  is characterized as the
 infimum of $k$ such that there exists a subsolution
 $f:M\times{\mS}^1\to \R$ of the Hamilton Jacobi equation,
 $$
 d_tf+H(x,d_xf,t)\le k.
 $$
 \end{theorem}
  
  We can recover the previous interpretation by using
  the autonomous Hamiltonian ${\mathbb H}(x,p,t,e)=H(x,p,t)+e$
  on $T^*(M\times{\mathbb S}^1)$. Then $df=(d_xf,d_tf)$ is an exact
  Lagrangian graph and $c(L)=\inf_u\sup_{(x,t)}{\mathbb H(d_{(x,t)}u)}$.
  The results in~\cite{gafa} can not be directly applied to this case
  because the Hamiltonian ${\mathbb H}$ does not come from a
  Lagrangian.

  The other values of Mather's alpha  function can be
  similarly characterized by recalling that $\a([\omega])=c(L-\omega)$
  and that the Hamiltonian of $L-\omega$ is
  $(x,p,t)\mapsto H(x,p+\omega,t)$.
 
  In corollary~\ref{Solk=c} we observe that differentiable solutions
  may only exist when $k=c(L)$. 
  
   \bigskip
   \begin{definition}
 Following Fathi~\cite{Fa1} we say that $u:M\times{\mS}^1\to \R$ is
 a forward weak KAM solution if
 \begin{enumerate}
 \item $u$ is $L+c$ dominated, i.e.
 $$
 u(y,[t])-u(x,[s])\le \Phi_c((x,[s]),(y,[t])).
 $$
 We use the notation $u\prec L+c$.
 \item For every $(x,[s])\in M\times\re$ there exists a curve
   $\ga:(s,\infty)\to M$ such that 
 $u(\ga(t),[t])-u(x,[s])= A_{L+c}(\ga|_{[s,t]})$,
 in that case we say that $\ga$ {\it realizes} $u$.
 
 \end{enumerate}
 
 Similarly  $u:M\times{\mS}^1\to \R$ is
 a backward weak KAM solution if it is dominated and 
 for every $(x,[s])\in M\times{\mS}^1$ there exists a curve
 $\ga:(-\infty,s)\to M$ such that 
 $u(x,[s])-u(\ga(t),[t])=A_{L+c}(\ga|_{[t,s]}))$
 
 Let $\cS^-$ (resp. $\cS^+$) be the set of {\it backward} (resp. {\it
 forward}) weak KAM solutions. 
 
 A point 
 $(x,v,[s])\in TM\times{\mS}^1$ is a positive (resp. negative)
 {\it semistatic}  point if the solution 
 $\ga=\ga_{(x,v,s)}$ of the Euler-Lagrange equation with
 initial conditions
 $(x,v,[s])$, satisfies for all $t$ 
 \[A_{L+c}(\ga|_{[s,t]})=\Phi_c((x,[s]),(\ga(t),[t]));\]
 (resp. $A_{L+c}(\ga|_{[t,s]})=\Phi_c((\ga(t),[t]),(x,[s]))$
 for all $t$).
 
 A point 
 $(x,v,[s])\in TM\times{\mS}^1$ is a {\it static} point if it is 
 positive semistatic and
 \[A_{L+c}(\ga|_{[s,t]})=-\Phi_c((\ga(t),[t]),(x,[s])).\]
 It turns out that if a point is static then its whole
 orbit under the Euler-Lagrange flow is static.
 
 We denote by $\Sigma^+$ the set of positive semistatic points.
 
 For a forward weak KAM solution $u$ we define its
 {\it forward basin} as
 \begin{multline*}
 \Ga_0^+(u)=\{(x,v,[s])\in\Sigma^+: \\
  u(\ga_{(x,v,s)}(t),[t])-u(x,[s])=
 \Phi_c((x,[s]),(\ga_{(x,v,s)}(t),[t]))\, \forall t> s\};
 \end{multline*}
 and define its {\it cut locus} by
 $\pi(\Ga_0^+(u)\setminus\Ga^+(u)\big)$, where
 $\pi:TM\times\mS^1\to M\times\mS^1$ is the proyection,
 \[\Ga^+(u)=\bigcup_{t>0}\phi_t\bigl(\Ga_0^+(u)\bigr),\]
 and $\phi_t$ is the Euler-Lagrange flow.
 It is easy to see that the sets $\Sigma^+$ and $\Ga_0^+(u)$ are 
 positively invariant and so $\Ga^+(u)\subset\Ga_0^+(u)$.
 Similarly, define the backward basins $\Ga^-_0(u)$, 
 $\Ga^-(u)$ for $u\in\cS^-$.
        \end{definition}

 The relevance of weak KAM solutions is that they have several properties, 
 including those given by the following theorem.
 
 \begin{theorem}\label{propiedades}
 
 If $u:M\times\mS^1\to\R$ is a weak KAM solution then
 
 \begin{enumerate}
 \item $u$ is Lipschitz and satisfies the Hamilton Jacobi equation
     \[H(x,d_xu,t)+d_tu=c\]
 at any point of differentiability. Moreover, $d_xu$ and $\dot\ga$
 are Legendre conjugate.
 \item Graph property: $\pi:\Ga^+(u)\to M\times\mS^1$ is injective and
   its inverse is Lipschitz.
 \item $u$ is differentiable on $\pi(\Ga^+(u))$.
 
 \end{enumerate}
 \end{theorem}
 
  Observe that since a weak KAM solution $u$ is Lipschitz,
  by Rademacher's theorem it is differentiable Lebesgue almost
  everywhere.
  
  Define the {\it Aubry set} $\cA$ as 
  $$
  \cA:=\{\,(x,[t])\in M\times\mS^1\,|\, h_c\big((x,[t]),(x,[t])\big)=0\,\}.
  $$ 
  We define an equivalence relation on $\cA$ by $(x,[s])\sim (y,[t])$ 
  if and only if 
 \[
 \Phi_c((x,[s]),(y,[t]))+\Phi_c((y,[t]),(x,[s]))=0. \]
 The equivalence classes of this relation are called
 {\it static classes}.
 
 Let $\fA$ be the set of static classes.
 For each static class $\Ga\in\fA$ choose a point $(p,[s])\in\Ga$ and
 let $\AA$ be the set of such points.

  \begin{remark}\label{Phi=h}
  
  Observe that by item~4 of proposition~\ref{barrera}, 
  if $(p,[\tau])\in\cA$ then
  $$
  h_c((p,[\tau]),(x,[t]))=\Phi_c((p,[\tau]),(x,[t])).
  $$
  \end{remark}
 
 \begin{theorem}\label{gonzalo}\quad
 The map $\{f:\AA\to\re\,|\, f\text{ dominated }\}\longrightarrow\cS^-$
 \begin{equation*}
 f\longmapsto  u_f(x,[t]) =\min\limits_{(p,[s])\in\AA}f(p,[s])+
 h_c((p,[s]),(x,[t])),
 \end{equation*}
 and the map $\{f:\AA\to\re\,|\, f\text{ dominated }\}\longrightarrow\cS^+$
 \begin{equation*}
 f\longmapsto  v_f(x,[t]) =\max\limits_{(p,[s])\in\AA}f(p,[s])
 -h_c((x,[t]),(p,[s])),
 \end{equation*}
 are bijections.
 \end{theorem}
 
 \section{The Peierls barrier}\label{Critical Value}
 
 We will be using the following lemma due to Mather~\cite{M}.
 We say that an absolutely continuous curve $\ga:[a,b]\to M$
 is a {\it minimizer } if 
 $A_L(\ga)\le A_L(\eta)$ for any absolutely continuous
 curve $\eta:[a,b]\to M$ with $\eta(a)=\ga(a)$ and $\eta(b)=\ga(b)$.
 It turns out that a minimizer is a solution of the Euler-Lagrange
 equation $\tfrac d{dt} L_v=L_x$.
 
 \begin{lemma}
 \label{le-mather}
 There is $A>0$ such that if $b-a\ge 1$ and 
 $\ga:[a,b]\to M$ is a minimizer, then $|\dot{\ga}(t)|\le A$ 
 for $t\in[a,b]$. 
 \end{lemma}
  
The proof of  most of Propositions~\ref{basic} and ~\ref{barrera} follow
standard arguments. We only give the proof of the Lipschitz continuity
of $h_c$.

  \begin{lemma}\label{hiswk} Given $(z,[\sigma])\in M\times {\Bbb S}^1$ define
  \[
  u(x,[t]) := h_c((z,[\sigma];x,[t]),\quad v(x,[t]):=-h_c((x,[t];z,[\sigma]).
  \]
  Then $u\in\cS^-$ and $v\in\cS^+$.
  \end{lemma}
  
  \begin{proof}
   By item 4 of proposition~\ref{barrera},
  $h(z,[\sigma]),(x,[t]))\le h_c((z,[\sigma]),(y,[s]))+\Phi_c((y,[s]),(x,[t]))$
  for all $(y,[s]),(x,[t])\in M\times {\Bbb S}^1$. Thus $u\prec L+c$.
  
  Given $(x,[t])\in M\times {\Bbb S}^1$ choose sequences $n_k\to+\infty$,
  $n_k\in\Z$ and $(x,v_k)\in T_xM$ such that
  $$
  h_c((z,[\sigma]),(x,[t]))=\lim_k
  A_{L+c}\bigl(\ga_k\vert_{[\sigma-n_k,t]}\bigr),
  $$
  where $\ga_k(s)=\pi\,\vr_{s-t}(x,v_k,t)$ is the solution of the
  Euler-Lagrange equation such that $(\ga_k(t),\dot\ga_k(t))=(x,v_k)$.
  By lemma~\ref{le-mather}, the norm $\lV v_k\rV$ is uniformly bounded.
  Choose a convergent subsequence $v_k\to w$. 
  Let $\eta(s):=\pi\,\vr_{s-t}(x,w,t)$, then for any fixed $s<0$,
  \begin{align*}
  h_c((z,[\sigma]),(x,[t]))
  &\le h_c((z,[\sigma]),(\eta(s),[s]))+A_{L+c}\bigl(\eta\vert_{[s,t]}\bigr)
  \\
  &=\lim_k h_c((z,[\sigma]),(\ga_k(s),[s]))
     +A_{L+c}\bigl(\ga_k\vert_{[s,t]}\bigr)
  \\
   &\le \lim_k A_{L+c}\bigl(\ga_k|_{[\sigma-n_k,s]}\bigr)
   +A_{L+c}\bigl(\ga_k|_{[s,t]}\bigr)
   \\
   &=h_c((z,[\sigma]),(x,[t])).
   \end{align*}
   So that
   $u(x,[t])-u(\eta(s),[s])=A_{L+c}\bigl(\eta|_{[s,t]}\bigr)$
    for all $s<0$.
  
  \end{proof}
  
  For autonomous lagrangians, dominated functions are Lipschitz.
  In contrast, for time periodic lagrangians the action potential
  is dominated but it is not continuous at $\big((x,s),(x,s)\big)$
  when $(x,s)$ is not in the Aubry set.
  Nevertheless, we have the following:
 
  \begin{lemma}\label{wk-lip}
  If $u:M\times\mS^1\to\R$ is a weak KAM solution
  (i.e. $u\in\cS^+\cup\cS^-$) then it is Lipschitz.
  Moreover the Lipschitz constant does not depend on $u$.
  \end{lemma}

  \begin{proof}
  Assume that $u\in\cS^-$, the case $u\in\cS^+$ is similar.
 Let $(x,[t_0]),(y,[s_0])\in M\times{\mS}^1$ be nearby points
 with $|s_0-t_0|<\tfrac 14$.
  Let $\ga:[0,1]\to M$ be a length minimizing geodesic joining $x$ to
  $y$ and let $\tau(r)=t_0+r\,(s_0-t_0)$, $r\in[0,1]$.
  Fix $\de>1$ and let $z:[t_0-\de,t_0]\to M$ be such that
  \begin{equation}\label{equ@r=0}
  u(x,[t_0])=u(z(t),[t])+\int_t^{t_0} L(z,\dot z)+c\;dt
  \quad \text{ for all }t_0-\de<t\le t_0.
  \end{equation}
  For $r\in[0,1]$, let $\eta(r,t)$, $t\in[t_0-\delta,\tau(r)]$, be a locally
  minimizing solution of~(E-L) such that $\eta(r,t_0-\delta)=z(t_0-\delta)$ and
  $\eta(r,\tau(r))=\ga(r)$.
  
  Then
  $$
  u(\ga(r),[\tau(r)])\le
  u(z(t_0-\delta),[t_0-\delta])
  +\int_{t_0-\delta}^{\tau(r)}L\bigl(\eta,\tfrac{\partial\eta}{\partial
  t},t\bigr)+c\;\;dt.
  $$
  with equality for $r=0$. Substracting the equality~\eqref{equ@r=0}
   at $r=0$, we get that
  \begin{equation}\label{Deltau}
  u(\ga(r),[\tau(r)])-u(x,[t_0])\le 
  \int_{t_0-\delta}^{\tau(r)}\big(L+c\big)\;dt
   -A_{L+c}(z|_{[t_0-\delta,t_0]}).
  \end{equation}
  Observe that this formula holds either for
  $s_0\le t_0$ or $t_0\le s_0$.
  As we shall see below, formula~\eqref{Deltau}
  implies that $u(y,s)-u(x,t)\le K\,\big(\, |s-t|+d(x,y)\,\big)$
  for some fixed $K>0$.
  Then changing the roles of $s$ and $t$ we get 
  that $u$ is Lipschitz.
  
  Indeed,
  differentiating the right hand side 
  and integrating by parts, we have 
  $$
  \frac{d\;}{dr}\int_{t_0-\delta}^{t_0+r\,(s_0-t_0)}
  L\bigl(\eta,\tfrac{\partial\eta}{\partial
  t},t\bigr)+c\;\;dt= $$
 $$=\left[L\bigl(\eta,\tfrac{\partial\eta}{\partial
  t},t\bigr)\big\vert_{(r,\tau(r))}+c\right](s_0-t_0)
  +\int_{t_0-\delta}^{\tau(r)} L_x\;\tfrac{\partial \eta}{\partial r}
  +L_v\;\tfrac{\partial^2\eta}{\partial t\partial r}
  $$
  $$=\left[L\bigl(\eta,\tfrac{\partial\eta}{\partial
  t},t\bigr)\big\vert_{(r,\tau(r))}+c\right](s_0-t_0)
  +\frac{\partial L}{\partial
  v}\bigl(\eta,\tfrac{\partial\eta}{\partial
  t},t)\bigr)\Big\vert_{(r,\tau(r))}
  \cdot\frac{\partial\eta}{\partial r}\Big\vert_{(r,\tau(r))}
  $$
 
  Observe that since $u$ is dominated the realizing curve $z$
  must be a minimizer.
  By lemma~\ref{le-mather},  $\lV\dot z\rV$ is uniformly bounded. 
  By the continuity of the solutions of~(E-L) with respect to initial
  values, $\lV\tfrac{\partial \eta}{\partial t}\rV$ is uniformly
  bounded. Hence there is a uniform constant $K>0$ (independent of
  $z(t),x,y,[s_0],[t_0], u$) such that 
  $$
  \lv L\bigl(\eta,\tfrac{\partial\eta}{\partial
  t},t\bigr)+c\rv\le K
  \qquad\text{ and }\qquad
  \lV\tfrac{\partial L}{\partial
  v}(\eta,\tfrac{\partial\eta}{\partial
  t},t)\rV<K.
  $$
  Since $\left.\dfrac{\partial\eta}{\partial
  r}\right\vert_{(r,\tau(r))}=\dot\ga(r)$,
  we get that
  \begin{align*}
  \frac{d\;}{dr}\left[\int_{t_0-\delta}^{\tau(r)} [L+c]\;\;
  -A_{L+c}(z)\right]
  \le K \; |s_0-t_0| + K\; \lV\dot\ga\rV. 
  \end{align*}
  The value of the right hand side of~\eqref{Deltau} is 0 at
  $r=0$.  Integrating this inequality,
  $$
  u(y,[s_0])-u(x,[t_0])\le K\;\big[ |s_0-t_0|+d(x,y)\big].
  $$
  Interchanging the roles of $(x,[t_0])$ and $(y,[s_0] )$ we obtain that
  the function  $u$ is Lipschitz.
  
  \end{proof}
  
    Combining lemmas~\ref{hiswk} and~\ref{wk-lip} 
    we get that the functions
    $f,g:M\times \mS^1\to\re$, $f(y,[t])=h_c((x,[s]),(y,[t]))$
    and $g(x,[s])=h_c((x,[s]),(y,[t]))$
    are Lipschitz. This implies that $h_c$ is Lipschitz. 
 \section{Subsolutions of the Hamilton Jacobi equation}
 \label{Hamilton Jacobi}
 
Following the same ideas as in~\cite{gafa}, one obtains
  \begin{lemma}
 \label{menor}
 If $k$ is a real number such that  there exists a function 
 $f$ in $C^1(M \times{\mS}^1)$ 
 subsolution of the Hamilton Jacobi equation 
 $$H(x,d_xf)+d_tf\le k$$
 Then $k \ge c(L)$.
 \end{lemma}
  
 \begin{lemma} 
 \label{dif}
 Let $k\geq c(L)$. If $f:M\times{\mS}^1\to\R$ is differentiable at 
 $(x,[t])\in M\times{\mS}^1$ and satisfies
 $$
 f(y,[t_2])-f(x,[t_1])\le\Phi_k (x,[t_1],y,[t_2])
 $$
 for all $y$ in a neighbourhood of $x$, then $H(x,d_{x}f)+d_tf\le k$.
 \end{lemma}
 
  \begin{proposition}
 For any $k>c(L)$ there 
 exists $f\in C^\infty(M\times{\mS}^1,\R)$ such that
 $H(x,d_xf,t)+d_tf<k$.                    \label{mayor}
 \end{proposition} 

 We give a proof of the following fact
  \begin{corollary}\label{Solk=c}
  If $u$ is a $C^{1+\text{\rm Lip}}$ global solution of the
  Hamilton-Jacobi equation
  $u_t+H(x,u_x,t)=k$, then $k=c(L)$ and $u$ is a weak KAM solution in
  ${\cS}^-\cap{\cS}^+$.
  \end{corollary}
  
  \begin{proof}
  By lemma~\ref{menor}, $k\ge c(L)$. Let
  ${\cL}_t(x,v)=L_v(x,v,t)$ be the conjugate 
  moment associated to $L$ and
  let $\xi(x,t)$ be the vector field defined by

   $\xi(x,t)={\cL}_t^{-1}(u_x)\in T_xM$.
   Then the vector field $(\xi,1)$ in $M\times{\mathbb S}^1$
   is Lipschitz.  Let $\rho_t$ be the flow of
   $(\xi,1)$ in $M\times{\mathbb S}^1$.
   From the Hamilton-Jacobi equation
   we get that
   \begin{equation}\label{diffdom}
   d_{(x,[t])}u\cdot (v,1) = u_x(x,t)\cdot v + u_t\cdot 1
   \le L(x,v,[t])+k.
   \end{equation}
   and that
   \begin{equation}\label{OnVectF}
   d_{(x,[t])}u\cdot(\xi(x,t),1)=L(x,\xi(x,t),t)+k
   \qquad\text{for all }(x,[t])\in M\times{\mathbb S}^1.
   \end{equation}
   Integrating equation~\eqref{diffdom} along absolutely continuous
   curves $(\ga(t),[t])$
   in $M\times{\mathbb S}^1$ from $(x,[s])$ to $(y,[t])$, we get that
   $$
   u(y,[t])-u(x,[s])\le\inf_\ga \oint_\ga (L+k)
   =\Phi_k\big((x,[s]),(y,[t])\big).
   $$ 
   So that $u\prec L+k$. 
   
   Also, integrating equation~\eqref{OnVectF},
   we get that the orbits of $\rho_t$ realize $u$ in the sense
   of the definition of a weak KAM solution.  
   In particular, the orbits of $\rho$ are global minimizers
   of the $(L+k)$-action, and thus they are solutions of the
   Euler-Lagrange equation. 
   
   It remains to prove that $k=c(L)$.
   Let $\nu$ be an invariant 
   Borel probability for $\rho_t$ and let $\mu$ be its lift to
   $TM\times{\mathbb S}^1$ using the vectorfield $\xi$. 
   Then $\mu$ is an invariant probability
   of the Lagrangian flow and, by equation~\eqref{OnVectF},
   $$
   \int(L+k)\;d\mu=\int du\;\; d\mu = 0.
   $$
   This implies that $k\le c(L)$.
   Thus $k=c(L)$ and also $\mu$ is a minimizing measure.
  \end{proof}

 \section{Weak KAM solutions}\label{weak KAM}
 
 \noindent{\bf Proof of theorem~\ref{propiedades}:}
 
 We first prove  item~1.
  By lemma~\ref{wk-lip}  we have   that 
  $u$ is Lipschitz and hence it is differentiable almost everywhere. Let
  $(x,[t])$ be a point of differentiability, then by lemma~\ref{dif} we have
 $$H(x,d_xu,t)+u_t \le c.$$
 Moreover let $\ga :[t,\infty)\to M$ be such that 
 $$
 u(\ga(s),[s])-u(x,[t])=A_{L+c}(\ga|_{[t,s]}),
 $$
 $$
 \lim_{s\to t} \frac{u(\ga(s),[s])-u(x,[t])}{s-t}
 = \lim_{s\to t} \frac{1}{s-t}
 \int_{t}^s(L+c)(\ga(s),\dot\ga(s),[s]) ds,
 $$
 so
 $$d_{x}u(x,[t])\dot\ga +d_tu(x,[t])= L(x,\dot\ga,t)+c.$$
 Therefore
 $$c=d_{x}u\,\dot\ga-L+d_tu \le H(x,d_{x}u,t)+d_tu\le c.$$
 
 So $u$ is a solution of the Hamilton Jacobi Equation and $d_{x}u$ and 
 $\dot\ga$ are related by the Legendre transformation of $L$.
 
 \qed

 \noindent{\bf Proof of the Graph Property:}

 We need the following lemma due to Mather, a proof of which can be
 found in~\cite{M}.
 
 \begin{lemma}
  \label{L:MATHER2}
   Given $A>0$  there exists
   $K>0$  $\e_1>0$ and  $\de>0$ with the following property: if
  $\lv v_i\rv <A$,
  $(p_i,v_i,[t_i])\in TM\times{\mS}^1$, $i=1,2$ satisfy
   $d((p_1,[t_1]),(p_2,[t_2]))<\de$ and
  $d((p_1,v_1,[t_1]),(p_2,v_2,[t_2]))\ge K^{-1}d((p_1,[t_1]),(p_2,[t_2]))$
   then, if  $a\in\re$ and
  $x_i:\re\to M$, $i=1,2$, are the solutions of $L$ with $x_i(t_i)=p_i$,
  $\dot{x}_i(p_i)=v_i$, there exist solutions $\ga_i:[t_i-\e,t_i+\e]\to M$
  of $L$ with $0<\e<\e_1$, satisfying
  \begin{align*}
  \ga_1(t_1-\e)=&x_1(t_1-\e) \quad , \quad \ga_1(t_2+\e)=x_2(t_2+\e) \, ,\\
  \ga_2(t_2-\e)=&x_2(t_2-\e) \quad , \quad \ga_2(t_1+\e)=x_1(t_1+\e) \, , \\
  S_L(x_1\vert_{[t_1-\e,t_2+\e]}&)
  +S_L(x_2\vert_{[t_2-\e,t_1+\e]})>S_L(\ga_1)+S_L(\ga_2)
  \end{align*}
  \end{lemma}
 We now prove the graph property. Let $(p_1,v_1,[t_1])$,~$(p_2,v_2,[t_2])
   \in\Ga^+(u)$
   and suppose that $K\,d((v_1,[t_1]),(v_2,[t_2]))>
   d((p_1,[t_1]),(p_2,[t_2]))$, where $K$ is
   from lemma~\ref{L:MATHER2} and the $A$ that we input 
   on lemma~\ref{L:MATHER2} is from
   lemma~\ref{le-mather}. Let $y^+_i=x_i(t_i+\e)$, $i=1,2$,
   and $y^-_i=x_i(t_i-\e)$ for $\e$ small,  then
 \begin{equation}
     \label{eq:diag1}
     u(y^+_1,[t_1+\e])-u(y^-_1,[t_1-\e])=
 \Phi_c((y^-_1,[t_1-\e]),(y^+_1,[t_1+\e]))
   \end{equation}
 \begin{equation}
     \label{eq:diag2}
    u(y^+_2,[t_2+\e])-u(y^-_2,[t_2-\e])=
 \Phi_c((y^-_2,[t_2-\e]),(y^+_2,[t_2+\e])) 
   \end{equation}
   Then using that $u\prec L+c$ and  lemma~\ref{L:MATHER2} , we get that
   \begin{multline}
     \label{eq:diag3}
     u(y^+_2,[t_2+\e])-u(y^-_1,[t_1-\e])+
     u(y^+_1,[t_1+\e])-u(y^-_2,[t_2-\e]) \\ 
 \le S_{L+c}(\ga_1)+S_{L+c}(\ga_2)\\
 <S_{L+c}(x_1\vert_{[t_1-\e,t_2+\e]}
  +S_{L+c}(x_2\vert_{[t_2-\e,t_1+\e]})\\
 =\Phi_c((y^-_1,[t_1-\e]),(y^+_1,[t_1+\e])+
 \Phi_c((y^-_2,[t_2-\e]),(y^+_2,[t_2+\e]). 
   \end{multline}
 
 Which is a contradiction with the sum of ~\eqref{eq:diag1} and
 ~\eqref{eq:diag2}. 
 \qed

  \bigskip

 \noindent{\bf Proof of item~3:}
 
 Let $(x,[s])$ in $\pi \Gamma^+(u)$, let 
 $(\sigma (\tau), [\alpha(\tau)])$
 be a curve on $M\times \mS^1$ with $\sigma (0)=x, \alpha (0)=s $. Let
 $\ga_s$ be the curve such that 
 
 $$u(\ga_s(t),[t])-
 u(\ga_s(s-\delta),[s-\delta])= A_{L+c}(\ga_s|_{[s-\delta,t]}).$$
 
 Since $(x,[s])$ is  in $\pi \Gamma^+(u)$ we can make a backwards
 variation $(\ga_\tau)$ of the solution $\ga_s$. That is,
 $ \ga_{\tau}:[s-\delta, \alpha(\tau)]\to M$ is a  solutions of
 the Euler -Lagrange equation joining the points $p=\ga_s(s-\delta)$ and
 $\sigma (\tau)$.

 Since $u $ is dominated we have
 
 \begin{eqnarray*}
   u(\sigma (\tau), [\alpha (\tau)])-u(x,[s]) & = & u(\sigma (\tau), [\alpha
 (\tau)])-u(p,[s-\delta])-(u(x,[s])-
 u(p,[s-\delta])) \\
  & \le &  A_{L+c}(\ga_{\tau}|_{[s-\delta,\alpha (\tau)]})-
  A_{L+c}(\ga_s|_{[s-\delta,s]}) \\
  & = & A_{L+c}(\ga_{\tau}|_{[s-\delta,s]})-
  A_{L+c}(\ga_s|_{[s-\delta,s]}) +A_{L+c}(\ga_{\tau}|_{[s,\alpha (\tau)]})
 \end{eqnarray*}

 Dividing by $\tau-s$ and taking limits as $\tau $ tends to $s$ and
 using the fact that $\ga_\tau$ is a solution of the Euler-Lagrange 
 equation, we obtain 
 
 $$\limsup_{\tau \to s} \dfrac{u(\sigma (\tau), [\alpha
   (\tau)])-u(x,[s])}
 {\tau-s}\le L_v(\dot \ga_s,s)\cdot\sigma'(0)+ L+c(\dot\ga_s,s)\alpha'(0)$$
 
 Similarly we can make a forward variation to get 
 
 $$\liminf_{\tau \to s} \dfrac{u(\sigma (\tau), [\alpha (\tau)])-u(x,[s])}
 {\tau-s}\ge L_v(\dot \ga_s, s)\cdot \sigma'(0)+ L+c(\dot \ga_s, s)\alpha'(0)$$
 
 \qed

 \noindent{\bf Proof of theorem~\ref{gonzalo}:}

  Let $u\in\cS^-$, since $u$ is dominated, then
  \begin{equation}\label{u<uh}
  u(x,[t])\le\min\limits_{(y,[r])}u(y,[r])+\Phi_c((y,[r]),(x,[t])).
  \end{equation}
  Let $\ga:]-\infty,t]\to M$ be such that for all $s\le t$,
  $$
  u(x,[t])-u(\ga(s),[s])=A_{L+c}\bigl(\ga\vert_{[s,t]}\bigr).
  $$
  Then $\ga(s)$ is semistatic and the minimum in~\eqref{u<uh}
  is realized at every point $(\ga(s),[s])$ with $s< t$.
  Choose a convergent sequence $(\ga(s_n),[s_n])\to (p,[\tau])\in
  M\times {\Bbb S}^1$, with $s_n\to-\infty$. Then by  lemma~\ref{AlSemi}
  below,
  $(p,[\tau])$ is in the Pierls set. Therefore, using the continuity
  of $\Phi_c$ at $(p,[\tau])$ (see lemma~\ref{PhiContP} below)
  and~\eqref{u<uh}, we have that
  \begin{align}
  u(x,[t])&=u(p,[\tau])+\Phi_c((p,[\tau]),(x,[t]))
          \notag\\
          &=\min_{(q,[\sigma])\in\cA}u(q,[\sigma])+\Phi_c((q,[\sigma]),(x,[t])).
          \label{u=uP}
  \end{align}

  We show now that it is enough to choose one point on each 
  static class to achieve the minimum on~\eqref{u=uP}.
  Suppose that $(p,[\tau])$ and $(q,[\sigma])$ are in the same
  static class. Then
  \begin{align*}
  \Phi_c((q,[\sigma]),(x,[t]))
  &\le\Phi_c((q,[\sigma]),(p,[\tau]))+\Phi_c((p,[\tau]),(x,[t]))\\
 &\le\Phi_c((q,[\sigma]),(p,[\tau]))+\Phi_c((p,[\tau]),(q,[\sigma]))
      +\Phi_c((q,[\sigma]),(x,[t]))\\
  &=\Phi_c((q,[\sigma]),(x,[t])).
  \end{align*}
  So that
  $\Phi_c((q,[\sigma]),(x,[t]))=\Phi_c((q,[\sigma]),(p,[\tau]))+\Phi_c((p,[\tau]),(x,[t]))$.
  Moreover,
  \begin{align*}
  u(p,[\tau]) &\le u(q,[\sigma]))+\Phi_c((q,[\sigma],p,[\tau])) \\
  &\le u(p,[\tau])+\Phi_c((p,[\tau]),(q,[\sigma]))+\Phi_c((q,[\sigma]),(p,[\tau]))\\
  &=u(p,[\tau]).
  \end{align*}
  So that $u(q,[\sigma])+\Phi_c((q,[\sigma]),(p,[\tau]))=u(p,[\tau])$. Thus
  \begin{align*}
  u(q,[\sigma])+\Phi_c((q,[\sigma]),(x,[t]))
   &= u(q,[\sigma])+\Phi_c((q,[\sigma]),(p,[\tau]))+\Phi_c((p,[\tau]),(x,[t])) \\
   &= u(p,[\tau])+\Phi_c((p,[\tau]),(x,[t])).
  \end{align*}
  So that $u=u_f$, with $f=u\vert_\AA$.
  
  Observe that by definition, if $f:\AA\to\re$ is dominated,
  then $u_f\vert_\AA\equiv f$. This implies that
  the map $\{f\text{ dominated}\}\mapsto u_f$ is injective.
  
  Finally, it remains to prove that if $f:\AA\to\re$ then
  $u_f\in\cS^-$. This follows from lemma~\ref{hiswk} 
  and lemma~\ref{minwk} below.
  \qed

  \begin{lemma}\label{AlSemi}
  If $\ga:]-\infty,t_0]\to M$ is semistatic and $s_n\to-\infty$
  is such that 
  $\lim_n\big(\ga(s_n),[s_n] \big)=(p,[\tau])$ exists.
  Then $(p,[\tau])$ is in the Aubry set.
  \end{lemma}
  
  \begin{proof}
  Let $\e>0$ be small. Chose $n_0>0$ such that for $n>n_0$, we have
  $$
  |s_n-\tau\mod1|<\tfrac\e2
  \qquad,\qquad
  d(\ga(s_n),p)<\tfrac\e2.
  $$
  Let $\la^-_n:[\tau,s_n+\e\mod1]\to M$ be a minimizer with 
  $\la^-_n(\tau)=p$, $\la^-_n(s_n+\e\mod1)=\ga(s_n+\e)$.
  By lemma~\ref{le-mather}, $\lV \dot\gamma\rV$ is uniformly
  bounded. By the same argument, using the first variation
  formula, as in proposition~\ref{barrera}.c,
  $$
  A_{L+c}(\la^-_n)\le K_1\,\big[\,
  d\big(\ga(s_n),p\big)+|s_n+\e-\tau\mod1|\,\big]
  \le 3\,\e\, K_1.
  $$
  Let $\la^+_n:[s_n-\e \mod 1,\tau]\to M$ be a minimizer with
  $\la^+_n(s_n-\e)=\ga(s_n-\e)$, $\la^+_n(\tau)=p$.
  Similarly,
  $$
  A_{L+c}(\la^+_n)\le 3\,\e\;K_1.
  $$
  We have that
  \begin{align}
  h_c\big((p,[\tau]),(p,[\tau])\big)
  &\le\liminf_{N\to\infty} A_{L+c}(\la^-_N)
    +A_{L+c}\big(\ga|_{[s_N+\e,s_n-\e]}\big)
    +A_{L+c}(\la^+_n)
    \notag\\
  &\le 6\,\e\,K_1
  +\liminf_N A_{L+c}\big(\ga|_{[s_N+\e,s_n-\e]}\big).
  \label{AlSemi.1}
  \end{align}
  
  Adding the action of $\ga$ on the intervals with
  endpoints $s_N-\e<s_N+\e<s_n-\e<s_n+\e$
  and using that $\ga$ is semistatic on 
  $[s_N-\e,s_n+\e]$, we have that
  \begin{align}\label{AlSemi.2}
  A_{L+c}\big(\ga|_{[s_N+\e,s_n-\e]}\big)
  =\Phi_c&\big((\ga(s_N-\e),s_N-\e)\,,\,(\ga(s_n+\e),s_n+\e)\big)
  \notag\\
  &-A_{L+c}\big(\ga|_{[s_N-\e,s_N+\e]}\big)
  -A_{L+c}\big(\ga|_{[s_n-\e,s_n+\e]}\big).
  \end{align}
  Comparing $\Phi_c$ with the action of a minimal length
  geodesic, parameterized by the {\it small} interval
  $I=[s_N-\e \mod1,\,s_n+\e \mod1]$
  of length $\e\le\ell(I)\le3\e$, with
  \begin{align*}
  \text{speed} 
  &\le \tfrac1\e\;d\big(\ga(s_N-\e),\ga(s_n+\e)\big)
  \le \tfrac 1\e\;\big[\,\e\,\lV\dot\ga\rV
  +d\big(\ga(s_N),\ga(s_n)\big)+\e\,\lV\dot\ga\rV\,\big]
  \\
  &\le 2\,\lV\dot\ga\rV+1;
  \end{align*}
  we have that
  $$
  \Phi_c\big((\ga(s_N-\e),s_N-\e)\,,\,
  (\ga(s_n+\e),s_n+\e)\big)
  \le\ell(I)\;\big[\,\max_{|v|\le 2\,\lV\dot\ga\rV+1}
  L+c\,\big]
  \le 3\,\e\,K_2.
  $$
  The two actions in~\eqref{AlSemi.2} are bounded by
  $2\,(2\e\cdot K_2)$. Thus, from~\eqref{AlSemi.2},
  $$
  A_{L+c}\big(\ga|_{[s_N+\e,s_n-\e]}\big)\le 7\,\e\,K_2.
  $$
  From~\eqref{AlSemi.1},
  $$
  0\le h_c\big((p,[\tau]),(p,[\tau])\big)
  \le 6\,\e\,K_1+7\,\e\,K_2.
  $$
  Now let $\e\to0$.
  \end{proof}

  \bigskip
  
   \begin{lemma}\label{PhiContP}
  If $\lim_n(y_n,[s_n])=(p,[\tau])\in\cA$ then for all 
  $(x,[t])\in M\times{\Bbb S}^1$,
  \begin{align*}
  \lim_n\Phi_c\big((y_n,[s_n]),(x,[t])\big)
  =\Phi_c\big((p,[\tau]),(x,[t])\big)
  =h_c\big((p,[\tau]),(x,[t])\big).
  \end{align*}
  \end{lemma}
  
  \begin{proof}
  
  Recall that by remark~\ref{Phi=h}, 
  $h_c((p,[\tau]),(x,[t]))=\Phi_c((p,[\tau]),(x,[t]))$.
  By item~4 of proposition~\ref{barrera},
  \begin{align}
  \Phi_c((p,[\tau]),(x,[t]))
    &=h_c((p,[\tau]),(x,[t])) \notag\\
    &\le h_c((p,[\tau]),(y_n,[s_n])) + \Phi_c((y_n,[s_n]),(x,[t]))
    \label{Phi=h.1}\\
    &\le h_c((p,[\tau]),(y_n,[s_n])) + h_c((y_n,[s_n]),(x,[t]))
    \notag\\
    &\le h_c((p,[\tau]),(y_n,[s_n])) + h_c((y_n,[s_n]),(p,[\tau]))
     + \Phi_c((p,[\tau]),(x,[t]))
     \label{Phi=h.2}
  \end{align}
  Using that $h_c$ is continuous, taking $\lim_n$ on inequalities
  ~\eqref{Phi=h.1} and~\eqref{Phi=h.2}, we get that
  $\lim_n\Phi_c((y_n,[s_n]),(x,[t]))=\Phi_c((p,[\tau]),(x,[t]))$.
  
  \end{proof}
  
  \bigskip

  \begin{lemma}\label{minwk}\quad
  
  If $\cU\subset\cS^-$, let  $\fu(x,[t]):=\inf_{u\in\cU}u(x,[t])$ then
  either $\fu\equiv-\infty$ or $\fu\in\cS^-$.
  
  If $\cV\subset\cS^+$, let  $\fv(x,[t]):=\sup_{v\in\cV}v(x,[t])$ then
  either $\fv\equiv+\infty$ or $\fv\in\cS^+$.
  \end{lemma}

  \begin{proof}
  Since $u\prec L+c$ for all $u\in\cU$, 
   for all $(x,[s]),(y,[t])\in M\times
  {\Bbb S}^1$,
  \begin{align}
  u(y,[t])&\le u(x,[s])+\Phi_c((x,[s]),(y,[t])), &&\text{ for all } u\in\cU,
  \notag \\
  \min_{u\in\cU}u(y,[t])=\fu(y,[t])&\le 
  u(x,[s])+\Phi_c((x,[s]),(y,[t])), &&\text{ for all } u\in\cU,
  \notag \\
  \fu(y,[t])&\le \fu(x,[s])+\Phi_c((x,[s]),(y,[t])).\label{fu-domi}
  \end{align}
  
  Now fix $(x,[t])\in M\times {\Bbb S}^1$ and fix a sequence $u_k\in \cU$
  such that $\fu(x,[t])=\lim_k u_k(x,[t])$. Let $(x,v_k,[t])\in \Ga^-(u_k)$.
  By lemma~\ref{le-mather}, $\lV v_k\rV$ is uniformly bounded.
  We can assume that $v_k\to w$. Let
  $\ga_{v_k}(s):=\pi\,\vr_{s-t}(x,v_k,t)$ and 
  $\ga_{w}(s):=\pi\,\vr_{s-t}(x,w,t)$. Then
  \begin{align*}
  u_k(x,t)=u_k(\ga_{v_k}(s),[s])+A_{L+c}\bigl(\ga_{v_k}|_{[s,t]}\bigr),
  &&\text{ for all }s<t,
  \end{align*}
  Since $\ga_{v_k}\overset{C^1}\longrightarrow \ga_w$ 
  uniformly on bounded intervals,
  using that by lemma~\ref{wk-lip} all the $u_k$'s have the
  same Lipschitz constant,
  taking the lim inf on $k$ we get that
  \begin{align}
  \fu(x,t)\ge \fu(\ga_{w}(s),[s])+A_{L+c}\bigl(\ga_{w}|_{[s,t]}\bigr),
  &&\text{ for all }s<t,\label{fu-realized}
  \end{align}
  The domination condition~\eqref{fu-domi} implies
  that~\eqref{fu-realized} is an equality.
  \end{proof}

 \end{document}